\documentclass[draft,oneside,11pt,reqno]{amsart}
\usepackage[a4paper,margin=30mm]{geometry}
\usepackage{amssymb}
\usepackage{amsxtra}
\usepackage[mathscr]{eucal}
\usepackage{graphics}
\usepackage{mathtools}
\usepackage[nobysame]{amsrefs}
\usepackage{color}
\usepackage{bm}

\makeatletter
\@namedef{subjclassname@2020}{\textup{2020} Mathematics Subject Classification}
\makeatother

\allowdisplaybreaks

\numberwithin{equation}{section}

\newtheorem{theorem}{Theorem}[section]

\newtheorem{corollary}[theorem]{Corollary}

\theoremstyle{definition}

\newtheorem{definition}[theorem]{Definition}

\newtheorem{examples}[theorem]{Examples}

\newtheorem{question}[theorem]{Question}

\renewcommand{\ge}{\geqslant}
\renewcommand{\le}{\leqslant}
\newcommand{\<}{\langle}
\renewcommand{\>}{\rangle} 
\renewcommand{\emptyset}{\varnothing}

\renewcommand{\setminus}{\smallsetminus}
\renewcommand{\SS}{\mathbb{S}}

\DeclareMathAlphabet{\mathsfit}{T1}{\sfdefault}{\mddefault}{\sldefault}
\SetMathAlphabet{\mathsfit}{bold}{T1}{\sfdefault}{\bfdefault}{\sldefault}

\newcommand{\ag}{\alpha_g}

\newcommand{\al}{\alpha}

\newcommand{\bfk}{\mathbf{k}}
\newcommand{\bfm}{\mathbf{m}}
\newcommand{\bfn}{\mathbf{n}}

\newcommand{\brm}{B_{R,M}}

\newcommand{\bx}{\mathbf{x}}
\newcommand{\brs}{B_R(\sss)}
\newcommand{\brt}{B_R(\sst)}
\newcommand{\bru}{B_R(\ssu)}
\newcommand{\CC}{\mathbb{C}}

\newcommand{\dimension}{\operatorname{dim}}
\newcommand{\dist}{\operatorname{dist}}

\renewcommand{\epsilon}{\varepsilon}

\newcommand{\fsh}{f^{\#}}
\newcommand{\fst}{f^{*}}

\newcommand{\lzdz}{\ell^1(\zd,\ZZ)}
\newcommand{\lzdr}{\ell^1(\zd,\RR)}
\newcommand{\lizdr}{\ell^\infty(\zd,\RR)}

\newcommand{\kgr}{\|\bfk\|\ge R}
\newcommand{\mgr}{\|\bfm\|\ge R}
\newcommand{\ngr}{\|\bfn\|\ge R}

\newcommand{\mzd}{\bfm\in\zd}
\newcommand{\nzd}{\bfn\in\zd}

\newcommand{\rd}{R_d}

\newcommand{\RR}{\mathbb R}

\newcommand{\sd}{\SS^{d}}
\newcommand{\sig}{\sigma}
\newcommand{\sbar}{\overline{\sig}}
\newcommand{\sg}{\sigma_g}

\newcommand{\ssf}{\mathcal{F}}
\newcommand{\ssp}{\mathcal{P}}
\newcommand{\sss}{\mathcal{S}}
\newcommand{\sst}{\mathcal{T}}
\newcommand{\ssu}{\mathcal{U}}
\newcommand{\supp}{\operatorname{supp}}

\newcommand{\TT}{\mathbb{T}}
\newcommand{\tzd}{\TT^{\zd}}

\newcommand{\U}{\mathsf{U}}

\newcommand{\wst}{w^{*}}

\newcommand{\xg}{X_g}
\newcommand{\ZZ}{\mathbb{Z}}

\newcommand{\zd}{\ZZ^{d}}

\newcommand{\zxd}{\ZZ[x_1^{\pm1},\dots,x_d^{\pm 1}]}

\begin{document}
\title[Divisibility and Lacunary Independence]{Divisibility of Integer Laurent Polynomials,
	\\
Homoclinic Points, and Lacunary Independence}

\author[Lind]{Douglas Lind}
\address{Douglas Lind: Department of Mathematics, University of
Washington, Seattle, Washington 98195, USA}
\email{lind@math.washington.edu}

\author[Schmidt]{Klaus Schmidt}
\address{Klaus Schmidt: Mathematics Institute,
University of Vienna, Oskar-Morgenstern-Platz 1, A-1090 Vienna, Austria}
\email{klaus.schmidt@univie.ac.at}

\dedicatory{Dedicated to the memory of Professor K.\ R.\ Parthasarthy}

\date{\today}




	\begin{abstract}
Let $f$, $p$, and $q$ be Laurent polynomials with integer coefficients in one or several variables, and suppose that $f$ divides $p+q$. We establish sufficient conditions to guarantee that $f$ individually divides $p$ and $q$. These conditions involve a bound on coefficients, a separation between the supports of $p$ and $q$, and, surprisingly, a requirement on the complex variety of $f$ called atorality satisfied by many but not all polynomials.

Our proof involves a related dynamical system and the fundamental dynamical notion of homoclinic point. Without the atorality assumption our methods fail, and it is unknown whether our results hold without this assumption.

We use this to establish exponential recurrence of the related dynamical system, and conclude with some remarks and open problems.
	\end{abstract}

\maketitle

\section{Introduction}
	\label{sec:introduction}

We begin with some notation and terminology, and then state our main results. In the next section we describe the dynamical context that underlies our proofs.

Let $d\ge1$, and let $\rd = \zxd$ denote the ring of Laurent polynomials with integer coefficients in $d$ commuting variables. We write $f\in\rd$ as
$f=\sum_{\nzd} f_{\bfn}\bx^{\bfn}$, where $\bx = (x_1,\dots,x_d)$, $\bfn=(n_1,\dots,n_d)$,
$\bx^{\bfn}= x_1^{n_1}\cdots x_d^{n_d}$,
and $f_{\bfn}\in\ZZ$ with $f_{\bfn}=0$ for all but finitely many $\bfn$. The units in $\rd$ are those elements of the form $\pm\, \bx^{\bfn}$. An element of $\rd$ is \emph{irreducible} provided that it is neither a unit nor a product of two non-units in ~$\rd$. We let $\<f\> = fR_d$ denote the principal ideal in $\rd$ generated by ~$f$. The \emph{support} of $f$ is defined as $\supp f = \{\nzd\colon f_{\bfn}\ne 0\}$, and we set $\|f\|_{\infty} = \max_{\nzd} |f_{\bfn}|$. The \emph{adjoint} $\fst$ of $f$ is the polynomial with coefficients defined by $\fst_{\bfn} = f_{-\bfn}^{}$. Let $\|\bfn\| = \max_{1\le j\le d}\{|n_j|\}$.
For nonempty subsets $\sss$ and $\sst$ of $\zd$ we define their distance to be $\dist(\sss,\sst) = \min_{\bfm\in\sss, \bfn\in\sst} \|\bfm-\bfn\|$.

Maps $v\colon\zd\to\RR$ are denoted by $v = (v_{\bfn})_{\nzd}$. We let $\lzdr$ be the space of all such maps $v$ for which $\|v\|_1 = \sum_{\nzd} |v_{\bfn}|<\infty$. This space is equipped with the convolution product
	\begin{equation}
	\label{eqn:convolution}
	(v\cdot w)_{\bfn}= \sum_{\mzd} v_{\bfm}w_{\bfn-\bfm} = \sum_{\mzd} v_{\bfn-\bfm}w_{\bfm} .
	\end{equation}
If we view each $f = \sum_{\nzd}f_{\bfn}\bx^{\bfn}\in R_d$ as the element $(f_{\bfn})_{\nzd}\in \lzdz\subset\lzdr$, we obtain an embedding $\rd \hookrightarrow \lzdr$ in which multiplication of Laurent polynomials agrees with the convolution product \eqref{eqn:convolution} in $\lzdr$.

In this paper we investigate divisibility of elements of $\rd$ by a given irreducible element $f\in\rd$. Somewhat surprisingly, our results hinge on a property of $f$ called \emph{atorality}. Let $\TT=\RR/\ZZ$ denote the additive torus, and $\SS = \{z\in\CC\colon |z|=1\}$ be its multiplicative counterpart.
	\begin{definition}
	\label{def:unitary}
The \emph{unitary variety} of $f\in\rd$ is defined to be
	\begin{displaymath}
\U(f) = \{(t_1,\dots,t_d) \in\TT^d \colon f(e^{2\pi i t_1},\dots,e^{2\pi i t_d})=0 \}.
	\end{displaymath}
	\end{definition}
\noindent
We remark that this is an additive version of the more usual definition of unitary variety as the intersection of the complex variety of $f$ with $\SS^d$.

According to the discussion in \cite{Lind-atoral}*{\S1}, the unitary variety $\U(f)$ (or, more accurately, its multiplicative counterpart) is a real semi\-algebraic set. By the cell decomposition theorem for such sets, it is a finite disjoint union of cells of various dimensions. We can therefore define its dimension $\dim \U(f)$ to be the maximal dimension of these cells, and this is known to be the same for all possible cell decompositions. If $\dim \U(f)=d$, then $f$ vanishes on an open subset of $\sd$, and by holding all variables but one fixed, it is easy to show that $f=0$.

	\begin{definition}
	\label{def:atoral}
A  Laurent polynomial $f\in\rd$ is \emph{atoral} if $\dimension \U(f)\le d-2$. This includes the possibility that $\U(f)=\emptyset$. If $\dim \U(f) = d-1$ then $f$ is called \emph{toral}.
	\end{definition}

These notions were introduced by Agler, Starkus, and McCarthy \cite{Agler-toral} in the context of several complex variables.

	\begin{examples}
	\label{exam:toral}
	(1) Every $f\in\rd$ with $\U(f)=\emptyset$ is atoral. The dynamical consequences of this case were extensively studied in \cite{Lind-homoclinic}.
	
	(2) If $f\in R_1$, then $f$ is atoral if and only if $\U(f)=\emptyset$, i.e., $f$ has no roots of absolute value 1. For instance, $x^2-x-1$ is atoral, while $x^4-x^3-x^2-x+1$ is toral \cite{Lind-atoral}*{Exam.\ 4.2}.
	
	(3) If $f(x,y)=1+x+y \in R_2$, then $\U(f)=\{(\frac{1}{3},\frac{2}{3}), (\frac{2}{3},\frac{1}{3})\}$ with dimension 0, and so $f$ is atoral.
	
	(4) If $f(x,y) = 3+x+y+x^{-1}+y^{-1}\in R_2$, then according to \cite{Lind-atoral}*{Exam. 4.2}
	\begin{displaymath}
	 \U(f) = \bigl\{(s,t)\in\TT^2: t = \pm \arccos\bigl( \tfrac{3}{2} - \cos 2\pi s\bigr), \ -\tfrac{1}{6}\le s\le\tfrac{1}{6} \bigr\},
	\end{displaymath}
	a smooth closed curve of dimension 1. Hence $f$ is toral.
	
	(5) If $f(x,y,z)=1+x+y+z\in R_3$, then $\U(f)$ is a union of three 1-dimensional circles \cite{Lind-atoral}*{Exam.\ 4.6}, and so $f$ is atoral.
	
	(6) A consequence of \cite{Lind-atoral}*{\S 2} is that if $f\in\rd$ is irreducible and if $\fst$ is not a unit times ~$f$, then $f$ is atoral. Thus it is very easy to create atoral examples.
	\end{examples}

Being atoral turns out to be equivalent to an invertibility property in $\lzdr$ established in \cite{Lind-atoral}*{Prop.~2.2 and Rem.~ 2.7}.

	\begin{theorem}
	\label{thm:quasi-inverse}
Let $0\ne f\in\rd$. Then $f$ is atoral if and only if there exists an element $\fsh\in\lzdr$ such that $\fsh \cdot f\in\rd\smallsetminus \<f\>$.
	\end{theorem}
\noindent
The proof consists of using atorality to find an $h\in\rd\smallsetminus \<f\>$ so that $h/f$ is smooth enough on $\SS^d$ to have absolutely convergent Fourier series, and then using the resulting Fourier coefficients to define $\fsh$. The element $\fsh$ is obviously not unique. We shall call any such $\fsh$ a \emph{quasi-inverse} of $f$.

For every nonempty subset $\sss\subset\zd$ and integer $H\ge1$ we let
	\begin{equation}
	\label{eq:PSH}
		\ssp(\sss, H) = \{p\in\rd\colon \supp(p)\subseteq\sss {\text{ and }} \|p\|_{\infty} \le H\}.
	\end{equation}

The following result was proved in \cite{Fan-bohr}*{Thm.\ 7.3}.

	\begin{theorem}[Gap Theorem]
	\label{thm:gap}
Let $f\in\rd$ be irreducible and atoral. For every $H\ge1$ there exists a number $M = M(f,H)\ge1$ with the following property. Suppose that $\sss$, $\sst$ are nonempty subsets of $\zd$ with $\dist(\sss,\sst)\ge M$, and that $p\in \ssp(\sss,H)$ and $q\in\ssp(\sst,H)$. If $f$ divides $p+q$, then $f$ divides $p$ and $f$ divides $q$.
	\end{theorem}

This theorem has two obvious corollaries.

	\begin{corollary}
	\label{cor:many-supports}
Let $f\in\rd$ be irreducible and atoral. For every $H\ge1$ there exists a number $M=M(f,H)\ge1$ with the following property. Suppose that $(\sss_i)_{i\in I}$ is a finite or infinite collection of nonempty subsets of $\zd$ with $\dist(\sss_i,\sss_j)\ge M$ whenever $i\ne j$. Let $F$ be an arbitrary finite subset of $I$ and let $p_i\in\ssp(\sss_i,H)$ for every $i\in F$. If $f$ divides $\sum_{i\in F} p_i$, then $f$ divides every $p_i$ for $i\in F$.
	\end{corollary}

The second corollary is closely related to the notion of lacunary independence ~\cite{Lind-lacunary} and its dynamical consequences. We discuss this further in ~\S\ref{sec:lacunary}.

	\begin{corollary}
	\label{cor:lacunary}
Let $f\in\rd$ be irreducible and atoral. For every finite non\-empty subset $\ssf\subset\rd$ there exists a number $M = M(f,\ssf)\ge1$ with the following property. For every finite subset $\sss\subset\zd$ with $\|\bfm-\bfn\|\ge M$ for $\bfm\ne\bfn$ in $\sss$, and for every choice of $p^{(\bfn)}\in\ssf$ for $\bfn\in\sss$, if $f$ divides the sum $\sum_{\bfn\in\sss}\bx^{\bfn} p^{(\bfn)}$ then $f$ divides each $p^{(\bfn)}$ for $\bfn \in\ssf$.
	\end{corollary}

Roughly speaking, this says that in the quotient ring $\rd/\<f\>$, multiplying a fixed finite subset by sufficiently widely spaced monomials results in independent subsets of the quotient ring as an additive group, hence the terminology lacunary independence.

Although these statements about divisibility of polynomials are purely algebraic, they arose in a dynamical context which we discuss in the next section.

\section{Algebraic $\zd$-actions}
	\label{sec:algebraic-actions}
An \emph{algebraic $\zd$-action} is an action of $\zd$ by continuous automorphisms of a compact metrizable abelian group. Our divisibility results are proved using a particular kind of algebraic $\zd$-action called a \emph{principal action}.

Let $\tzd$ denote the compact abelian group consisting of all $t=(t_{\bfn})_{\nzd}$ with $t_{\bfn}\in\TT$ for every $\bfn\in\zd$. Define the \emph{shift action} $\sig$ on $\tzd$ by $(\sig^{\bfm}t)_{\bfn}=t_{\bfn+\bfm}$. Every $g = \sum_{\nzd}\in \rd$ defines a homomorphism $g(\sig)=\sum_{\mzd}g_{\bfn}\sig^{\bfn}$ from $\tzd$ to itself. Fix such a $g$, and consider the closed, shift-invariant subgroup $\xg$ of $\tzd$ defined by
	\begin{equation}
	\label{eq:xg}
\xg = \ker g(\sig) = \Bigl\{ t\in\tzd\colon \sum_{\nzd}t_{\bfm+\bfn} g_{\bfn} \equiv 0 \pmod 1 \text{ for all $\bfm\in\zd$}\Bigr\}.
	\end{equation}
Let $\ag$ denote the restriction of $\sig$ to $\xg$. The dynamical system $(\xg,\ag)$ is called the \emph{principal algebraic $\zd$-action} corresponding to $g\in\rd$. When $g=0$ this reduces to $\xg=\tzd$ and $\ag=\sig$.

In order to use analytical ideas, we linearize $(\xg,\sg)$ as follows. Let $\lizdr$ denote the space of all $v=(v_{\bfn})$ for which $v_{\bfn}\in \RR$ and $\|v\|_{\infty}=\sup_{\nzd}|v_{\bfn}|<\infty$. Define the shift action $\sbar$ on $\lizdr$ by $(\sbar^{\bfm} v)_{\bfn} = v_{\bfn+\bfm}$. For every $w\in\lzdr$ we can define a bounded shift-commuting linear operator $w(\sbar)$ on $\lizdr$ by $w(\sbar)=\sum_{\mzd}w_{\bfm}\sbar^{\bfm}$. Let $\wst$ denote the adjoint element $\wst_{\bfn}= w_{-\bfn}^{}$. Then
	\begin{displaymath}
\bigl( w(\sbar)v\bigr)_{\bfn} = \sum_{mzd} w_{\bfm} (\sbar^{\bfm} v)_{\bfn}=
\sum_{\mzd}w_{\bfm}v_{\bfn+\bfm} = (v\cdot\wst)_{\bfn}.
	\end{displaymath}
Define the surjective homomorphism $\eta\colon\lizdr\to\tzd$ by $\eta(v)_{\bfn}= v_{\bfn} \pmod 1$. If $w\in\rd$, then $\eta\circ w(\sbar)= w(\sig)\circ \eta$.

An element $t\in\xg$ is called a \emph{summable homoclinic point} if there exists a $v\in\lzdr$ such that $\eta(v)=t$. Such points provide a powerful way to localize behavior in algebraic $\zd$-actions (see \cite{Lind-homoclinic} and \cite{Lind-atoral} for many applications), and are the essential ingredient to prove the results in \S\ref{sec:introduction}.

\section{Proof of the Gap Theorem}
	\label{sec:proof}

Since the proof of Theorem \ref{thm:gap} appearing in \cite{Fan-bohr}*{Theorem 7.3} is quite short, but still illustrates the effective use of dynamical ideas in proving algebraic results, we include it here for the reader's convenience.

	\begin{proof}[Proof of Theorem \ref{thm:gap}]
Suppose that $f\in\rd$ is irreducible and atoral. Put $g=\fst$, and define the principal algebraic $\zd$-action $(\xg,\ag)$ as above. By Theorem ~\ref{thm:quasi-inverse} we may fix a quasi-inverse $\fsh\in\lzdr$ of $f$ such that $\fsh\cdot f = h\in\rd\smallsetminus \<f\>$. Then
	\begin{displaymath}
0 = \eta(h) = \eta(\fsh\cdot f) = \eta(g(\sbar)\cdot \fsh) = g(\sig)\eta(\fsh),
	\end{displaymath}
so that $\eta(\fsh)\in\xg$. If $\eta(\fsh)$ were equal to zero, then $\fsh$ would be a polynomial, violating Theorem \ref{thm:quasi-inverse}. Hence $\eta(\fsh)$ is a nonzero summable homoclinic point in $\xg$.

Since $\fsh\in\lzdr$, for every $H\ge 1$ there is an $R = R(f,\fsh, H)$ such that
	\begin{displaymath}
\sum_{\ngr}|\fsh_{\bfn}|<\frac{1}{2H\| f\|_1} .
	\end{displaymath}
For every nonempty subset $\sss\subset\zd$, let $B_R(\sss) = \{\bfn\in\zd\colon \dist(\bfn,\sss)\le R\}$. We put $M=3R$, and show that this separation satisfies the conclusion of the theorem.

Suppose that $\sss$ and $\sst$ are nonempty subsets of $\zd$ with $\dist(\sss,\sst)\ge M=3R$, and let $\ssu=\sss\cup\sst$. Let $p\in\ssp(\sss,H)$ and $q\in\ssp(\sst,H)$, and put $r = p+ q\in\ssp(\ssu,H)$. We assume that $f$ divides $r$, and will prove that $f$ divides both $p$ and $q$ individually.

\smallskip Since $f$ divides $r$, there is a $k\in\rd$ such that $r=k\cdot f$. Then
	\begin{displaymath}
r\cdot \fsh = (k\cdot f)\cdot \fsh = k\cdot(f\cdot \fsh) = k\cdot h\in \rd .
	\end{displaymath}
Furthermore, we claim that
	\begin{displaymath}
\supp(r\cdot \fsh) \subset \bru = \brs \cup \brt.
	\end{displaymath}
For if $\bfn\notin \bru$ then $r_{\bfn-\bfm}=0$ whenever $\|\bfm\|\le R$. Hence
	\begin{displaymath}
|(r\cdot\fsh)_{\bfn}|=\Bigl| \sum_{\mzd} r_{\bfn-\bfm}\fsh_{\bfm}\Bigr| \le \|r\|_{\infty} \sum_{\mgr}|\fsh_{\bfm}| < H\cdot\frac{1}{2H\|f\|_1} \le \frac{1}{2}.
	\end{displaymath}
Since $r\cdot\fsh = k\cdot h\in\rd$ has integer coefficients, this shows that $\supp(r\cdot\fsh)\subset \bru$.

\smallskip Now let $u$ be the restriction of the integer polynomial $r\cdot\fsh$ to $\brs$, so that $u\in\rd$ as well. We claim that $\|u-p\cdot\fsh\|_{\infty}<\frac{1}{2\|f\|_1}$.

Indeed, if $\bfn\in\brs$, then $\dist(\bfn,\sst)>R$, so that
	\begin{displaymath}
\Bigl| (u-p\cdot\fsh)_{\bfn}\Bigr| = \Bigl|(q\cdot\fsh)_{\bfn}\bigr|=
\Bigl|\sum_{\bfm\in\sst}q_{\bfm}^{} \fsh_{\bfn-\bfm}\Bigr|
\le \|q\|_{\infty} \sum_{\kgr} |\fsh_{\bfk}| < H\cdot\frac{1}{2H\|f\|_1} .
	\end{displaymath}
On the other hand, if $\bfn\notin\brs$, then
	\begin{displaymath}
|(u-p\cdot\fsh)_{\bfn}| = |(p\cdot\fsh)_{\bfn}| = \Bigl| \sum_{\bfm\in\sss}
p_{\bfm}^{}\fsh_{\bfn-\bfm} \Bigr|
\le \|p\|_{\infty}\sum_{\kgr} |\fsh_{\bfk}|
< H\cdot\frac{1}{2H\|f\|_1}.
	\end{displaymath}
Observe that both $p\cdot \fsh\cdot f = p\cdot h$ and $u\cdot f$ lie in $\rd$. Furthermore,
	\begin{displaymath}
\|p\cdot \fsh\cdot f - u\cdot f\|_{\infty}\le
\|p\cdot\fsh - u\|_{\infty}\cdot \|f\|_1 < \frac12.
	\end{displaymath}
Hence $(p\cdot\fsh-u)\cdot f=0$. Thus $p\cdot h = p\cdot \fsh\cdot f = u\cdot f$. Since $f$ is irreducible and does not divide $h$, it must divide $p$. Then $f$ must also divide $q=r-p$.
	\end{proof}

In order to interpret Corollary \ref{cor:lacunary} for our purposes we first consider the trivial (but instructive) case $f=0$. In this case the Pontryagin dual $\widehat{X_g}$ of the group $X_g = \mathbb{T}^{\mathbb{Z}^d}$ in \eqref{eq:xg} is equal to $R_d$ (cf. \cite{Schmidt}*{Example 5.2 (1)}): every character (i.e., continuous group homomorphism) $\chi \colon X_g\to \mathbb{T}$ of $X_g$ is of the form
	\begin{equation}
	\label{eq:character}
t\mapsto \chi _p(t) = \sum_{\mathbf{n}\in \mathbb{Z}^d}t_\mathbf{n}p_\mathbf{n}\pmod 1
	\end{equation}
for some $p=\sum_{\mathbf{n}\in \mathbb{Z}^d}p_\mathbf{n}\mathbf{x}^\mathbf{n}\in R_d$. If $p$ and $q$ are elements of $\rd$ with disjoint supports, the corresponding characters are independent of each other; in particular, if $p+q=0$, then both $p$ and $q$ are equal to zero.

\smallskip Now assume that $f\ne 0$. In this case the Pontryagin dual of the group $X_g$ in \eqref{eq:xg} is equal to $R_d/\langle f \rangle $ (\cite{Schmidt}*{Example 5.2 (2)}): every character $\chi \in \widehat{X_g}$ is again of the form $\chi _p\colon X_g\to \mathbb{T}$ in \eqref{eq:character} for some $p=\sum_{\mathbf{n}\in \mathbb{Z}^d}p_\mathbf{n}\mathbf{x}^\mathbf{n}\in R_d$, but $\chi _p = \chi _q$ whenever $p-q \in \langle f \rangle $. In particular, if $p\in \langle f \rangle $, then the corresponding character of $X_g$ is trivial.

\smallskip Although the notion of ``support'' no longer has an obvious meaning for characters of $X_g$, the map $p\mapsto \eta(p \cdot \fsh)$ maps an entire coset of $p$ in $\rd/\<f\>$ to a single summable homoclinic point of $X_g$. By bounding the coefficients of $p$ by $H$ and using $R = R(f,\fsh,H)$ from the proof, most of the $\ell^1$-mass of $p\cdot \fsh$ is located within $B_R(\supp p)$, and we can think of this set as the ``approximate support'' of $p$. If the supports of two polynomials $p$ and $q$ are widely separated, then so are their approximate supports, and multiplication of $p\cdot \fsh$ and $q\cdot \fsh$ by $f$ leads to Laurent polynomials $p'=p\cdot \fsh\cdot f$ and $q'=q\cdot \fsh\cdot f$ with genuinely disjoint supports, in analogy with the trivial case.

\section{Lacunary Independence and Exponential Recurrence}
	\label{sec:lacunary}

In this section we discuss some dynamical implications of the results in \S\ref{sec:introduction}. Let $\mathcal{A}$ be a finite set, $\nu_0$ be a probability measure on $\mathcal{A}$, $Y=\mathcal{A}^{\zd}$ equipped with the product topology, and $\nu=\nu_0^{\zd}$ be the product measure on $Y$. Then the $\zd$-shift $\sigma_Y$ on $Y$ preserves $\nu$, and $(Y,\nu,\sigma_Y)$ is called a \emph{Bernoulli $\zd$-shift}.

An algebraic $\zd$-action $\al$ on $X$ automatically preserves the normalized Haar measure $\mu$ on $X$. The action $\al$ is called \emph{Bernoulli} if there is a Bernoulli $\zd$-shift $(Y,\nu,\sigma_Y)$ and a measure-preserving equivariant map $\pi\colon X\to Y$. If $\pi$ can be chosen to be continuous off a null set of $X$, then $\al$ is called \emph{finitarily Bernoulli}. For instance, a hyperbolic toral automorphism generates a $\ZZ$-action that is finitarily Bernoulli since it has Markov partitions \cite{Keane-Smorodinsky}. However, if the toral automorphism has eigenvalues of modulus 1, then although known to be Bernoulli, it is not known whether it is finitarily Bernoulli. These cases correspond to toral and atoral polynomials in one variable, respectively.

We address the question of whether analogues of this exist for algebraic $\zd$-actions. Our starting point is the fact that if $f\in\rd$ is irreducible, then $\al_f$ is Bernoulli except for the special cases of so-called generalized cyclotomic polynomials (see \cite{Schmidt}*{\S23} and \cite{Rudolph-Schmidt}).

For an integer $R\ge 1$, let $B_R=\{\bfn\in\zd:\|\bfn\|\le R\}$. For an integer $M\ge 1$, let $\brm=B_R\cap M\zd$. Let $|\cdot|$ denote cardinality, and observe that
	\begin{equation}
	\label{eqn:balls}
B_R\subset \brm + B_M, \text{\quad and hence \quad} |\brm|\ge \frac{|B_R|}{|B_M|}.
	\end{equation}
As observed in \cite{Keane-Smorodinsky}, the following condition is necessary for an algebraic action to be finitarily Bernoulli.

	\begin{definition}
	\label{def:exponential}
An algebraic $\zd$-action $\al$ on $(X,\mu)$ is called \emph{exponentially recurrent} provided that for every proper closed subset $K$ of $X$ there are constants $c>0$ and $0<\theta<1$ such that for all $R\ge1$ we have that
	\begin{equation}
	\label{eqn:theta}
\mu\Bigl( \bigcap_{\bfn\in B_R} \al^{\bfn}K\Bigr)\le c\,\theta^{|B_R|}.
	\end{equation}
	\end{definition}

The following shows that this condition is satisfied for some principal actions.

	\begin{theorem}
	\label{thm:exponential}
Let $f\in\rd$ be irreducible and atoral. Then the principal algebraic $\zd$-action $\al_f$ is exponentially recurrent with respect to the normalized Haar measure $\mu_f$ on $X_f$.
	\end{theorem}

	\begin{proof}
Having already established lacunary independence in Corollary \ref{cor:lacunary}, this proof follows the path in \cite{Lind-lacunary}*{\S2} where $d=1$.

Simplify notation by putting $X=X_f$, $\al=\al_f$, and $\mu=\mu_f$. Let $K$ be a proper closed subset of $X$. Then $\mu(K)<1$. Fix $\eta$ with $\mu(K)<\eta<1$.

Let $\epsilon>0$ to be determined later. There is an open neighborhood $U$ of $K$ with $\mu(U\setminus K)<\epsilon$. By Urysohn's lemma \cite{E-W}*{Lemma\ A.27} there is a continuous function $r$ on $X$ such that $0\le r\le 1$ on $X$, $r=1$ on $K$, and $r=0$ on $X\setminus U$. By the Stone--Weierstrass theorem \cite{E-W}*{Thm.\ 2.40}, there are a finite subset $\mathcal{E}\subset\zd$, a positive integer $H$, and an $\mathbb{R}$-linear combination $\phi $ of the characters $\{\chi _p:p\in \mathcal{F}\coloneqq \mathcal{P}(\mathcal{E},H)\}\subset \widehat{X_g}$ (defined as in \eqref{eq:PSH} and \eqref{eq:character}) such that $\|r-\phi\|_{C(X)} <\epsilon$. Let $\psi=\phi+\epsilon$. Then $0\le\psi\le 1+\epsilon$ on $X$, $\psi\ge1$ on $K$, and $\psi<\epsilon$ on $X\setminus U$. Hence
	\begin{displaymath}
\int_X \psi\,d\mu \le (1+\epsilon)\mu(K) + (1+\epsilon)\mu(U\setminus K)+\epsilon \mu(X\setminus U) <\eta
	\end{displaymath}
provided that $\epsilon$ is small enough.

By Corollary \ref{cor:lacunary}, there is an $M=M(f, \mathcal{F})$ such that for every $R\ge1$ the sets $\{\bx^{\bfn}\mathcal{F}\colon \bfn\in\brm\}$ are independent in the sense described there. This independence shows that the constant term in the product $\prod_{\mathbf{n}\in \brm}\bx^{\bfn}\psi$ is the product of the constant terms of the factors, which are all equal to $\int_X \psi\,d\mu$. Hence using \eqref{eqn:balls},
	\begin{displaymath}
	\begin{aligned}
\mu\Bigl(\bigcap_{\bfn\in B_R} \al^{\bfn}K\Bigr) & \le \mu\Bigl(\bigcap_{\bfn\in\brm} \al^{\bfn}K\Bigr)\le \int_X \Bigl(\prod_{\bfn\in\brm} \bx^{\bfn}\psi\Bigr)\,d\mu
	\\
&= \prod_{\bfn\in\brm} \int_X \bx^{\bfn} \psi\,d\mu < \eta^{|\brm|}\le
\eta^{|B_R|/|B_M|} = \theta^{|B_R|},
	\end{aligned}
	\end{displaymath}
where $\theta = \eta^{1/|B_M|}$ satisfies condition \eqref{eqn:theta} for an appropriate constant $c>0$.
	\end{proof}

We remark that \cite{Lind-lacunary} proves exponential recurrence when $d=1$ even for toral polynomials. For instance, if $f(x) = 5x^2 -6x+5$ then both complex roots of $f$ have absolute value 1, but there is a different argument using its $5$-adic roots.

\section{Remarks and Questions}
	\label{sec:remarks}

The proof of the Gap Theorem \ref{thm:gap} depends crucially on the assumption that $f$ is atoral, thereby providing a nonzero summable homoclinic point in the associated dynamical system. If $f$ is toral it is known that no such points exist \cite{Lind-growth}*{Prop.\ 2.2}, and this method of proof fails from the very beginning. Ironically, for the toral $f$ in Example \ref{exam:toral} (4), the associated dynamical system has uncountably many points whose coordinates tend to 0 at ~$\infty$, but not fast enough to be summable (these points are provided by the Fourier coefficients of measures supported on the positively curved unitary variety $\U(f)$, see \cite{Lind-homoclinic}*{Exam.\ 7.3}).

The possible extensions of results depending on atorality to the toral case typically involve delicate --- and at present unresolved --- Diophantine questions. For instance, the main result in \cite{Lind-atoral}*{Thm.\ 1.3} on the convergence of the number of periodic components to entropy uses summable homoclinic points. A completely different proof has been recently given in \cite{Dimitrov-galois}*{App.\ B} using algebraic number theory to give effective rates of convergence, and in earlier work \cite{Dimitrov-mahler} this has been partially extended to the toral case using intricate Diophantine arguments.

	\begin{question}
	\label{question:toral}
Is the Gap Theorem valid for irreducible toral polynomials in $\rd$?
	\end{question}

Lacunary independence can be formulated quite generally. Say that a finite collection $\{E_j\colon 1\le j \le n\}$ of finite subsets of an abelian group is \emph{independent} if whenever $e_j^{}, e_j' \in E_j$ for all $j$ and $\sum_{j=1}^{n} e_j^{} = \sum_{j=1}^{n} e_j'$, then $e_j^{}=e_j'$ for $1\le j\le n$. If this holds, then the cardinality of the sum of the $E_j$ is the product of their cardinalities. Call an $\rd$-module $A$ \emph{lacunarily independent} if for every finite subset $F\subset A$ there is an $M=M(F,A)$ such that if $\sss$ is a finite subset of $\zd$ with $\|\bfm-\bfn\|\ge M$ for distinct $\bfm,\ \bfn\in\sss$, then the sets $\{\bx^{\bfn}\cdot F\colon \bfn\in\sss\}$ are independent.
Corollary \ref{cor:lacunary} shows that if $f$ is irreducible and atoral, then $\rd/\<f\>$ is lacunarily independent as an $\rd$-module.

	\begin{question}
	\label{question:module}
Which $\rd$-modules are lacunarily independent? 
	\end{question}

Some conditions beyond those in the Gap Theorem are clearly necessary. For instance, let     $\mathfrak p$ be the prime ideal  $\<2, 1+x+y\>$ in $R_2$. Then in the quotient module  $R_2/\mathfrak{p}$, the Frobenius squaring map shows that $(1+x+y)^{2^n}=1+x^{2^n}+y^{2^n}$ for all $n\ge1$. Hence $1+x+y$ divides sums of monomials with unbounded gaps, and so $R_2/\mathfrak{p}$ violates lacunary independence using the set $F =\{0,1\}$.

The construction of dynamical actions from Laurent polynomials described in \S\ref{sec:introduction} can be carried out for every $R_d$-module $A$, resulting in an action $\alpha_A$ by automorphisms of a compact abelian group $X_A$ (the Pontryagin dual of $A$) which preserves the normalized Haar measure $\mu_A$ (see \cite{Schmidt}*{Chap.\ II}). The entropy of $\alpha_A$ with respect to $\mu_A$ coincides with its topological entropy, and we denote their common value by $h(\al_A)$. 

We are indebted to Hanfeng Li for his comments leading to the following consequences of lacunary independence.

Peters \cite{Peters} showed how to compute the entropy of an automorphism of a compact abelian group using the dual automorphism of the discrete dual group. Chung and Li \cite{Chung-Li} pointed out that this idea extends to $\zd$-actions (and even to actions by countable amenable groups), and the details are provided in \cite{Kerr-Li}*{\S13.8}.

Now suppose that $A$ is a nonzero $\rd$-module that is lacunarily independent. Let $F$ be a subset of $A$ with two elements. The proof of Theorem~\ref{thm:exponential} shows that there are $c>0$ and $\gamma>0$ such that 
\begin{displaymath}
   \biggl|\,\sum_{\bfn\in B_R} \bx^{\bfn}\cdot F \, \biggr|\ge c\, 2^{\gamma\,|B_R|}
\end{displaymath}
for all $R\ge1$. By the extension of \cite{Peters} mentioned above, we can conclude that $h(\alpha_A)>0$. Furthermore, every nonzero submodule $B$ of $A$ is obviously lacunarily independent as well, so that $h(\alpha_B)>0$. From this and \cite{Schmidt}*{Thm.\ 20.8} it follows that $\alpha_A$ has completely positive entropy, and hence is Bernoulli \cite{Schmidt}*{\S23}. In particular, if $\mathfrak a$ is a nonprincipal ideal in $\rd$, then $h(\alpha_{\rd/{\mathfrak a}})=0$ by \cite{Schmidt}*{Prop.\ 17.5} and so $\alpha_{\rd/{\mathfrak a}}$ is not lacunarily independent (for instance the case $\<2,1+x+y\>$ discussed above, with a more direct argument).

Although the question of which $\rd$-modules correspond to finitarily Bernoulli actions is still open, even for $d=1$, the additional assumption of lacunary independence may provide an approach.

\begin{question}
   If $A$ is a lacunarily independent $\rd$-module, is $\alpha_A$ finitarily Bernoulli?
\end{question}

	\end{document}